\documentclass[a4paper,10pt]{amsart}
\usepackage{epsfig}
\usepackage{graphicx,color}
\usepackage{amsmath,amsfonts,amssymb,eufrak}
\usepackage{float}
\usepackage{graphics}
\usepackage{latexsym}
\usepackage{float}
\usepackage{verbatim}
\usepackage{xy}
\xyoption{matrix}
\xyoption{arrow}
\input xy 
\usepackage{rotating}
\xyoption{all}         

\newtheorem{lem}{Lemma}[section]
\newtheorem{prop}[lem]{Proposition}
\newtheorem{cor}[lem]{Corollary}
\newtheorem{thm}[lem]{Theorem}

\newtheorem*{thmnn}{Theorem}

\newcommand{\Ext}{\operatorname{Ext}\nolimits}

\newcommand{\mo}{\operatorname{mod}\nolimits}

\newcommand{\Hom}{\operatorname{Hom}\nolimits}
\newcommand{\End}{\operatorname{End}\nolimits}
\newcommand{\add}{\operatorname{add}\nolimits}
\newcommand{\Dim}{\operatorname{dim}\nolimits}

\begin{document}
\title{Finite mutation classes of coloured quivers}
\author{Hermund Andr\' e Torkildsen}

\begin{abstract} We consider the general notion of coloured quiver
  mutation and show that the mutation class of a coloured quiver $Q$,
  arising from an $m$-cluster tilting object associated with $H$, is
  finite if and only if $H$ is of finite or tame representation type,
  or it has at most $2$ simples. This generalizes a result known for
  $1$-cluster categories.
\end{abstract}

\maketitle

%%%%%%%%%SECTION: INTRODUCTION

\section*{Introduction} 

Mutation of skew-symmetric matrices, or equivalently quiver mutation,
is very central in the topic of cluster algebras \cite{fz}. Quiver 
mutation induces an equivalence relation on the set of quivers. The 
mutation class of a quiver $Q$ consists of all quivers mutation
equivalent to $Q$. In \cite{br} it was shown that the mutation class
of an acyclic quiver $Q$ is finite if and only if the underlying graph 
of $Q$ is either Dynkin, extended Dynkin or has at most two vertices.

Cluster categories were defined in \cite{bmrrt} in the general case and
in \cite{ccs} in the $A_n$-case as a categorical model of the
combinatorics of cluster algebras. Some cluster categories have a nice
geometric description in terms of triangulations of certain polygons,
see \cite{ccs, s}. This was used in \cite{to,bto} to count the number
of quivers in the mutation classes of quivers of Dynkin type $A$ and 
$D$. In \cite{brs} they used different methods to count the number of
quivers in the mutation classes of quivers of type $\widetilde{A}$.

A generalization of cluster categories, the $m$-cluster categories,
have been investigated by several authors. See for example
\cite{bm1,bm2,bt,iy,k,t,w,z,zz}. In \cite{bt} 
mutation on coloured quivers was defined, and we can define mutation 
classes of coloured quivers. It is a natural question to ask when the 
mutation classes of coloured quivers are finite. In this paper we 
want to show the following theorem, analogous to the main theorem in 
\cite{br}. 

\begin{thmnn}
Let $k$ be an algebraically closed field and $Q$ a connected finite
quiver without oriented cycles. The following are equivalent for $H =
kQ$.
\begin{enumerate}
\item There are only a finite number of basic $m$-cluster tilted algebras
  associated with $H$, up to isomorphism. 
\item There are only a finite number of Gabriel quivers occurring for
  $m$-cluster tilted algebras associated with $H$, up to isomorphism.
\item $H$ is of finite or tame representation type, or has at most two
  non-isomorphic simple modules.
\item There are only a finite number of $\tau$-orbits of cluster
  tilting objects associated with $H$. 
\item There are only a finite number of coloured quivers occurring for
  $m$-cluster tilting objects associated with $H$, up to isomorphism.
\item The mutation class of a coloured quiver $Q$, arising from an
  $m$-cluster tilting object associated with $H$, is finite. 
\end{enumerate}
\end{thmnn}

%%%%%%%%%SECTION: BACKGROUND

\section{Background}
Let $H=kQ$ be a finite dimensional hereditary algebra over an
algebraically closed field $k$, with $Q$ a quiver with $n$ vertices. 
The cluster category was defined in
\cite{bmrrt} and independently in \cite{ccs} in the $A_n$
case. Consider the bounded derived category $\mathcal{D}^{b}(H)$ of $\mo H$. Then the cluster category is defined as the orbit
category $\mathcal{C}_H = \mathcal{D}^b(H)/\tau^{-1}[1]$, where $\tau$
is the Auslander-Reiten translation and $[1]$ is the shift functor. 

As a generalization of cluster categories, we can consider the
$m$-cluster categories defined as $\mathcal{C}_{H}^{m} =
\mathcal{D}^b(H)/\tau^{-1}[m]$. The $m$-cluster category was shown in
\cite{k} to be triangulated. The $m$-cluster
category is a Krull-Schmidt category, an $(m+1)$-Calabi-Yau category,
and it has an AR-translate $\tau = [m]$. The indecomposable objects in
$\mathcal{C}_H^m$ are of the form $X[i]$, with $0 \leq i < m$, where
$X$ is an indecomposable $H$-module, and of the form $P[m]$, where $P$
is a projective $H$-module. 

An $m$-cluster tilting object is an object $T$ in $\mathcal{C}_H^m$
with the property that $X$ is in $\add T$ if and only if
$\Ext_{\mathcal{C}_H^m}^i(T,X)=0$ for all $i \in \{1,2,...,m\}$. It
was shown in \cite{w,zz} that an object which is maximal $m$-rigid,
i.e. it has the property that $X \in \add T$ if and only if
$\Ext_{\mathcal{C}_H^m}^i(T \oplus X,T \oplus X)=0$ for all $i \in
\{1,2,...,m\}$, is also an $m$-cluster tilting object. They also
showed that an $m$-cluster tilting object $T$ always has $n$
non-isomorphic indecomposable summands.    

An almost complete $m$-cluster tilting object $\bar{T}$ is an object
with $n-1$ non-isomorphic indecomposable direct summands such that
$\Ext_{\mathcal{C}_H^m}^i(\bar{T},\bar{T})=0$ for $i \in
\{1,2,...,m\}$. It is known from \cite{w,zz} that any almost complete
$m$-cluster tilting object has exactly $m+1$ complements, i.e. there
exist $m+1$ non-isomorphic indecomposable objects $T'$ such that
$\bar{T} \oplus T'$ is an $m$-cluster tilting object. 

Let $\bar{T}$ be an almost complete $m$-cluster tilting object and
denote by $T_k^{(c)}$, where $c\in \{0,1,2,...,m\}$, the complements
of $\bar{T}$. In \cite{iy} it is shown that the complements are
connected by $m+1$ exchange triangles 

$$T_k^{(c)} \rightarrow B_k^{(c)} \rightarrow T_k^{(c+1)}\rightarrow,$$
where $B_k^{(c)}$ are in $\add{\bar{T}}$.

An $m$-cluster tilted algebra is an algebra of the form
$\End_{\mathcal{C}_H^m}(T)$, where $T$ is an $m$-cluster tilting
object in $\mathcal{C}_H^m$.

%%%%%%%%%SECTION: COLOURED QUIVER MUTATION

\section{Coloured quiver mutation}

In the case when $m=1$ there is a well-known procedure for the exchange 
of indecomposable direct summands of a cluster-tilting object. Given an
almost complete cluster-tilting object, there exist exactly two complements,
and the corresponding quivers are given by quiver mutation. For an arbitrary
$m \geq 1$, the procedure is a little more complicated. Since 
an almost complete $m$-cluster tilting object has, up to isomorphism, exactly 
$m+1$ complements, the Gabriel quiver does not give enough 
information to keep track of the exchange procedure. Buan and Thomas therefore 
defined a class of coloured quivers in \cite{bt}, and they define a mutation 
procedure on such quivers to model the exchange on $m$-cluster tilting objects. 
In this section we recall some results from this paper.

To an $m$-cluster tilting object $T$, Buan and Thomas associate a
coloured quiver $Q_T$, with arrows of colours chosen from the set 
$\{0,1,2,...,m\}$. For each indecomposable summand of $T$ there is
a vertex in $Q_T$. If $T_i$ and $T_j$ are two indecomposable summands
of $T$ corresponding to vertex $i$ and $j$ in $Q_T$, there are $r$
arrows from $i$ to $j$ of colour $c$, where $r$ is the multiplicity of
$T_j$ in $B_i^{(c)}$. 

They show that such quivers have the following properties.
\begin{enumerate}
\item The quiver has no loops.
\item If there is an arrow from $i$ to $j$ with colour $c$, then there exist no arrow
from $i$ to $j$ with colour $c' \neq c$.
\item If there are $r$ arrows from $i$ to $j$ of colour $c$, then there are $r$
arrows from $j$ to $i$ of colour $m-c$.
\end{enumerate}

They also define coloured quiver mutation, and they give an algorithm
for the procedure. Let $Q=Q_T$, for an $m$-cluster tilting object $T$,
be a coloured quiver and let $j$ be a vertex in $Q$. The mutation of
$Q$ at vertex $j$ is a quiver $\mu_j(Q)$ obtained as follows. 

\begin{enumerate}
\item For each pair of
  arrows $$\xymatrix{i\ar[r]^{(c)}&j\ar[r]^{(0)}&k\\}$$ where $i \neq
  k$ and $c \in \{0,1,...,m\}$, add an arrow from $i$ to $k$ of colour
  $c$ and an arrow from $k$ to $i$ of colour $m-c$. 
\item If there exist arrows of different colours from a vertex $i$ to
  a vertex $k$, cancel the same number of arrows of each colour until
  there are only arrows of the same colour from $i$ to $k$. 
\item Add one to the colour of all arrows that goes into $j$, and
  subtract one from the colour of all arrows going out of $j$.
\end{enumerate}

See Figure \ref{figmutationexample} for an example.

  \begin{figure}[htp]
  \begin{center}
    $\xymatrix{1 \ar@/^/[r]^{(0)} & 2 \ar@/^/[l]^{(2)}\ar@/^/[r]^{(0)}&3\ar@/^/[l]^{(2)}}
\stackrel{\mu_1}{\longrightarrow}
    \xymatrix{1 \ar@/^/[r]^{(2)} & 2 \ar@/^/[l]^{(0)}\ar@/^/[r]^{(0)}&3\ar@/^/[l]^{(2)}}
\stackrel{\mu_1}{\longrightarrow}
    \xymatrix{1 \ar@/^/[r]^{(1)} & 2 \ar@/^/[l]^{(1)}\ar@/^/[r]^{(0)}&3\ar@/^/[l]^{(2)}}$
  \end{center}\caption{\label{figmutationexample} Examples of mutation
    of coloured quivers for Dynkin type A and $m=2$.}  
  \end{figure}

In \cite{bt} the following theorem is proved.

\begin{thm}
Let $T = \oplus_{i=1}^{n} T_i$ be an $m$-cluster tilting object in
$\mathcal{C}_H^m$. Let $T' = T / T_j \oplus T_j^{(1)}$ be an
$m$-cluster tilting object where there is an exchange triangle $$T_j
\rightarrow B_j^{(0)} \rightarrow T_j^{(1)} \rightarrow.$$ Then
$Q_{T'} = \mu_j(Q_T)$. 
\end{thm}

The quiver obtained from $Q_T$ by removing all arrows of colour
different from $0$ is the Gabriel quiver of the $m$-cluster tilted
algebra $\End_{\mathcal{C}_H^m}(T)$. Quivers of $m$-cluster tilted
algebras can be reached by repeated coloured quiver mutation
\cite{zz} (see also \cite{bt}). 

\begin{prop}\label{tiltreached}
Any $m$-cluster tilting object can be reached from any other
$m$-cluster tilting object via iterated mutation.
\end{prop}

They obtain the following corollary.

\begin{cor}
For an $m$-cluster category $\mathcal{C}_H^m$ of the acyclic quiver
$Q$, all quivers of $m$-cluster tilted algebras are given by repeated 
mutation of $Q$.
\end{cor}

Let us always denote by $Q_G$ the Gabriel quiver of the coloured
quiver $Q$. In this paper we are only interested in coloured quivers
which arises from an $m$-cluster tilting object. Let $Q_G$ be an
acyclic quiver and $Q$ the coloured quiver obtained from $Q_G$ by
adding the necessary arrows of colour $m$, i.e. if there exist $r$
arrows from $i$ to $j$ of colour $0$, then add $r$ arrows from $j$ to
$i$ of colour $m$. Then the quivers which arises from $m$-cluster
tilting objects are exactly the quivers mutation equivalent to $Q$.    

Let $Q$ be a coloured quiver with arrows only of colour $0$ and $m$,
as above, and where the underlying graph of the Gabriel quiver $Q_G$ is of
Dynkin type $\Delta$. Then certainly $Q_G$ is a quiver of an
$m$-cluster tilted algebra. Let us call the set of quivers mutation
equivalent to $Q$ the mutation class of type
$\Delta$. Certainly, all orientations of $\Delta$ (as a Gabriel
quiver) is in the mutation class of type $\Delta$.    

Figure \ref{figmutationclass} shows all non-isomorphic coloured
quivers in the mutation class of type $A_3$ for $m=2$.

  \begin{figure}[htp]
  \begin{center}
    $$\xymatrix{1 \ar@/^/[r]^{(0)} & 2
      \ar@/^/[l]^{(2)}\ar@/^/[r]^{(0)}&3\ar@/^/[l]^{(2)}} \quad\quad\quad
    \xymatrix{1 \ar@/^/[r]^{(2)} & 2 \ar@/^/[l]^{(0)}\ar@/^/[r]^{(0)}&3\ar@/^/[l]^{(2)}}$$
    $$\xymatrix{1 \ar@/^/[r]^{(1)} & 2
      \ar@/^/[l]^{(1)}\ar@/^/[r]^{(0)}&3\ar@/^/[l]^{(2)}}\quad\quad\quad
\xymatrix{1 \ar@/^/[r]^{(0)} & 2
  \ar@/^/[l]^{(2)}\ar@/^/[r]^{(1)}&3\ar@/^/[l]^{(1)}}$$
$$\xymatrix{1 \ar@/^/[r]^{(0)} & 2
      \ar@/^/[l]^{(2)}\ar@/^/[r]^{(2)}&3\ar@/^/[l]^{(0)}}\quad\quad\quad
\xymatrix{1 \ar@/^/[r]^{(1)} & 2
  \ar@/^/[l]^{(1)}\ar@/^/[r]^{(1)}&3\ar@/^/[l]^{(1)}}$$
$$\xymatrix{1 \ar@/^/[r]^{(1)} \ar@/_2.4pc/[rr]_{(0)} & 2
      \ar@/^/[l]^{(1)}\ar@/^/[r]^{(1)}&3\ar@/_2.4pc/[ll]_{(2)}\ar@/^/[l]^{(1)}}$$
  \end{center}\caption{\label{figmutationclass} All non-isomorphic coloured quivers
    in the mutation class of $A_3$ for $m=2$.}   
  \end{figure}
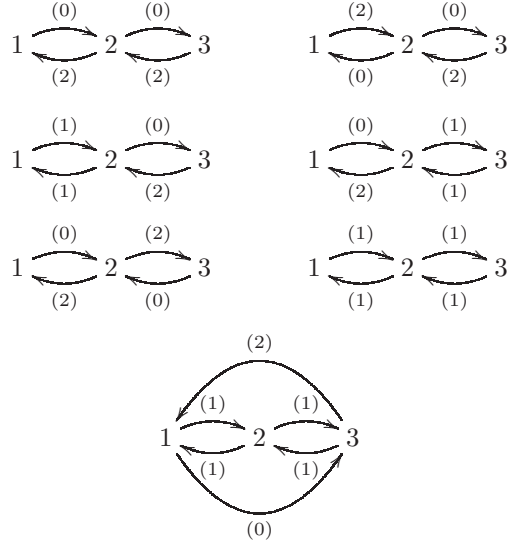

We note that in a mutation class, there can be several
non-isomorphic coloured quivers with the same underlying Gabriel
quiver, and that the Gabriel quiver of an $m$-cluster tilted algebra
might be disconnected.

To any $m$-cluster tilting object $T$ there exist a coloured quiver 
$Q_T$, but we also have the following. 

\begin{lem}\label{colouredquivertilting}
Suppose $Q$ is a coloured quiver in some mutation class of a
quiver of an $m$-cluster tilted algebra. Then there exist an
$m$-cluster tilting object $T$ such that $Q = Q_T$.   
\end{lem}
\begin{proof}
This follows directly from the corollary, since mutation of $m$-cluster
tilting objects corresponds to mutation of coloured quivers.
\end{proof}

We know that $[i]$ is an equivalence on the $m$-cluster category for
all integers $i$. In particular, $\tau = [m]$ is an equivalence.

\begin{prop}\label{colouredquiveriso}
If $T$ is an $m$-cluster tilting object, then $Q_T$ is isomorphic to 
$Q_{T[i]}$ for all $i$
\end{prop}
\begin{proof}
It is enough to prove that $Q_T$ is isomorphic to $Q_{T[\pm 1]}$. Suppose
there are $r$ arrows in $Q_T$ from $i$ to $j$ with colour $c$. Let
$T_i$ and $T_j$ be the indecomposable direct summands of $T$
corresponding to vertex $i$ and $j$ in $Q_T$ respectively. Let $\bar{T} = T /
T_i$ be the almost complete $m$-cluster tilting object obtained from
$T$ by removing $T_i$. Then there exist an exchange triangle 
$$T_i^{(c)} \rightarrow B_i^{(c)} \rightarrow T_i^{(c+1)}\rightarrow$$
with $B_i^{(c)}$ in $\add(\bar{T})$. There are $r$ arrows from
$i$ to $j$, with colour $c$, so hence $T_j$ has multiplicity $r$ in
$B_i^{(c)}$. Clearly $T_i[1]$ and $T_j[1]$ are indecomposable
direct summands of $T[1]$ and we have the exchange triangle 
$$T_i^{(c)}[1] \rightarrow B_i^{(c)}[1] \rightarrow
T_i^{(c+1)}[1]\rightarrow.$$  
Since $T_j$ has multiplicity $r$ in $B_i^{(c)}$, $T_j[1]$ has
multiplicity $r$ in $B_i^{(c)}[1]$. It follows that there are $r$
arrows in $Q_{T[1]}$ from $i$ to $j$ with colour $c$. The same proof
holds for $[-1]$, and so hence the claim follows. 
\end{proof}

\section{Finiteness of the number of non-isomorphic $m$-cluster tilted
  algebras} 

In \cite{br} the authors showed that if $Q$ is a finite quiver with no
oriented cycles, then there is only a finite number of quivers in the
mutation class of $Q$ if and only if the underlying graph of $Q$ is
Dynkin, extended Dynkin or has at most two vertices. In these cases
there are only a finite number of non-isomorphic cluster-tilted
algebras of some fixed type. In this section we want to prove an
analogous result for coloured quivers by generalizing the results and
proofs in \cite{br}.     

Let $H = kQ$ be a finite dimensional hereditary algebra. We know that
$H$ is of finite representation type if and only if the underlying
graph of $Q$ is Dynkin. Furthermore, $H$ is tame if and only if the
underlying graph of $Q$ is extended Dynkin. Objects in the module
category of $H$, when $H$ is of infinite type, are either
preprojective, preinjective or regular. In the case when $H$ is tame,
the regular components of the AR-quiver are disjoint tubes of the form
$\mathbb{Z}A_{\infty} / \left <\tau^i\right >$ for some $i$, and in the wild
case they are of the form $\mathbb{Z}A_{\infty}$.

If $X$ is a preprojective or preinjective $H$-module, it is known that
$X$ is rigid, i.e. $\Ext_{H}^1(X,X)=0$. The following is a
well-known result, see for example \cite{r}. 

\begin{lem}\label{kQ1}
Let $H=kQ$ be a finite dimensional hereditary algebra of infinite
representation type, then if $H$ has exactly two simples, no
indecomposable regular object is rigid. 
\end{lem}

In \cite{w} it was shown that if $T$ is an $m$-cluster tilting object
in $\mathcal{C}_{H}^{m}$, then it is induced from a tilting object in 
$\mo H_0 \vee \mo H_0 [1] \vee ... \vee \mo H_0 [m-1]$,
where $H_0$ is derived equivalent to $H$. If $H$ is of finite or tame
representation type, it was shown in \cite{br}
that for each indecomposable projective $H$-module $P$, there are only
a finite number of indecomposable objects $X$ such that
$\Ext_{\mathcal{C}_H^1}^1(X,P)$. 

\begin{lem}\label{projectiveshift}
Let $P[i]$ be a shift of an indecomposable projective $H$-module,
where $H$ is of finite or tame representation type. Then there is only
a finite set of objects $X$ in $\mathcal{C}_{H}^{m}$ with
$\Ext^{k}_{\mathcal{C}_{H}^{m}}(X,P[i]) = 0$ for all $k \in
\{1,2,...,m\}$.  
\end{lem}
\begin{proof}
We can assume that an $m$-cluster
tilting object is induced from a tilting object in $\mo H \vee \mo H
[1] \vee ... \vee \mo H [m-1]$. 

It is enough to show that there are a finite number of indecomposable
objects $X$ such that $\Ext^{1}_{\mathcal{C}_H^m}(X,P)=0$, where $P$ is
a projective $H$-module, since the shift functor is an equivalence on
the $m$-cluster category. It follows from \cite{br} that there are
only a finite number of indecomposable objects $X$ lying inside $\mo
H[i]$, with $\Ext^{1}_{\mathcal{C}_H^m}(X,P[i]) = 0$ for all $i$.  

We have $\Ext^{j+1}_{\mathcal{C}_H^m}(X,P) =
\Ext^{1}_{\mathcal{C}_H^m}(X,P[j])$, so there are only finitely many
indecomposable objects $X$ in $\mo H[j]$ such that
$\Ext^{j+1}_{\mathcal{C}_H^m}(X,P) = 0$. Consequently there are only a
finite number of indecomposable objects $X$ such that
$\Ext^{k}_{\mathcal{C}_H^m}(X,P) = 0$ for all $k \in \{1,2,...,m\}$,
and we are finished.    
\end{proof}

It is known from \cite{bkl} that in the tame case, a collection of one
or more tubes is triangulated. We give the proof of the following for
the convenience of the reader.   

\begin{prop}\label{thm_regular}
Let $H$ be a finite dimensional tame hereditary algebra over a
field $k$, and $\mathcal{C}_H^{m}$ the corresponding $m$-cluster
category. Let $$X \rightarrow Y \rightarrow Z \rightarrow$$ be a
triangle in $\mathcal{C}_H^m$, where two of the terms are shifts of
regular modules. Then all terms are shifts of regular modules.   
\end{prop}
\begin{proof}
It is enough to show that if $X$ and $Z$ are shifts of regular modules, then
$Y$ is a shift of a regular module. There exist a homogeneous tube
$\mathcal{T}$, i.e. $\tau M = 
M$ for all $M \in \mathcal{T}$, such that no
direct summands of $X$ or $Z$ are in $\mathcal{T}$. Let $W$ be a
quasi-simple object in $\mathcal{T}$. We have that $W$ is sincere (see
\cite{dr}). We get the exact sequence $$\Hom(Z,W) \rightarrow
\Hom(Y,W) \rightarrow \Hom(X,W).$$ We have that $\Hom(Z,W) = \Hom(X,W)
= 0$, since there are no maps between disjoint tubes. It follows that
$\Hom(Y,W) = 0$. Since $W$ is sincere, we have that $\Hom(U,W) \neq 0$
for any projective $U$, hence for any preprojective since $\tau W =
W$. We can do similarly for preinjectives. It follows that all direct
summands of $Y$ are shifts of regulars.
\end{proof}

\begin{prop}\label{thm_lem}
Let $\mathcal{C}_{H}^{m}$ be an $m$-cluster category, where $H$ is of
tame representation type. Let $T$ be an $m$-cluster tilting object in
$\mathcal{C}_{H}^{m}$. Then $T$ has, up to $\tau$, at least one direct
summand which is a shift of a projective or injective.
\end{prop}
\begin{proof}
It is clearly enough to prove that there are no $m$-cluster tilting
objects in $\mathcal{C}_{H}^{m}$ with only shifts of regular
$H$-modules as direct summands. So suppose, for a contradiction, that
such a $T$ exists. 

We can decompose $T$ into indecomposable summands, where $T=T_1 \oplus
T_2 \oplus ... \oplus T_n$ and $n$ is the number of simple
$H$-modules. If all direct summands are of the same degree, we already
have a contradiction, since a tilting module has at least one direct
summand which is preprojective or preinjective (see \cite{r}).

Assume that $T_n$ is a direct summand of degree $k \leq m$. Let
$\bar{T} = T_1 \oplus T_2 \oplus ... \oplus T_{n-1}$ be the almost
complete $m$-cluster tilting object obtained from $T$ by removing the
direct summand $T_n$. Then we know that the complements of $\bar{T}$
are connected by $m+1$ AR-triangles,

$$M_{i+1} \rightarrow X_i \rightarrow M_i \rightarrow,$$
where $i \in \{0,1,2,...,m \}$ and $X_i \in \add \bar{T}$. 

The direct summands of $X_i$ are by assumption shifts of regular
modules. We also have that $T_n$ is a shift of a regular module and
that it is equal to $M_j$ for some $j$, since it is a complement of
$\bar{T}$. It follows that $M_i$ is a shift of a regular module for
all $i$ by Proposition \ref{thm_regular}, since these are connected by the
exhange triangles. So all $m$-cluster tilting objects that can be
reached from $T$ by a finite number of mutations, have only regular
direct summands.  

This leads to a contradiction, because we know from Proposition
\ref{tiltreached} that all $m$-cluster tilting objects can be reached
from $T$ by a finite number of 
mutations, and a tilting module in $H$ induces an $m$-cluster tilting
object in $\mathcal{C}_{H}^{m}$ with at least one direct summand
preprojective or preinjective.
\end{proof}

From this it follows that we can assume that an $m$-cluster tilting object
has at least one direct summand which is a shift of a projective up to
$\tau$. 

We also need a lemma proven in \cite{br}.

\begin{lem}\label{brlemma1}
Let $H$ be wild with at least $3$ non-isomorphic simples. Let $t$ be
a positive integer. Then there is a tilting module $T$ in $H$ with
indecomposable direct summands $T_1$ and $T_2$, such that $\Dim
\Hom_{H} (T_1, T_2) \geq t$. 
\end{lem}

To prove the next lemma, which was observed in \cite{br} for
$1$-cluster tilted algebras, we use the following fact from
\cite{w}. Let $F = \tau^{-1}[m]$. If $X$ and $Y$ are two objects in
some chosen fundamental domain in $\mathcal{D}^b(H)$, then
$\Hom_{\mathcal{D}^b(H)}(X,F^i Y) = 0$ for all $i \neq 0,1$. 

\begin{lem}\label{brlemma2}
If a path in the quiver of an $m$-cluster tilted algebra goes through two
oriented cycles, then it is zero.
\end{lem}
\begin{proof}
We have that $$\Hom_{\mathcal{C}_H}(X,Y)= \oplus_{i\in
  \mathbb{Z}}\Hom_{\mathcal{D}^b(H)}(X,F^i Y).$$ Let $X$ and $Y$ be
two indecomposable $m$-rigid objects in a chosen fundamental
domain. It is well known that since $\Ext_{\mathcal{D}^b(H)}(X,X)=0$,
we have that $\End_{\mathcal{D}^b(H)}(X) = k$. It follows that in an
oriented cycle, one of the maps lifts to a map of the form $X
\rightarrow F Y$ in $\mathcal{D}^b(H)$. If there is a path that goes
through two oriented cycles, we have a map of the form $X \rightarrow
FY \rightarrow F^2 Z$, and this is $0$ by the above. 
\end{proof}

The following theorem generalizes the main theorem in \cite{br}.

\begin{thm}\label{main1}
Let $k$ be an algebraically closed field and $Q$ a connected finite
quiver without oriented cycles. The following are equivalent for $H =
kQ$.
\begin{enumerate}
\item There are only a finite number of basic $m$-cluster tilted algebras
  associated with $H$, up to isomorphism. 
\item There are only a finite number of Gabriel quivers occurring for
  $m$-cluster tilted algebras associated with $H$, up to isomorphism.
\item $H$ is of finite or tame representation type, or has at most two
  non-isomorphic simple modules.
\item There are only a finite number of $\tau$-orbits of cluster
  tilting objects associated with $H$.
\item There are only a finite number of coloured quivers occurring for
  $m$-cluster tilting objects associated with $H$, up to isomorphism.
\item The mutation class of a coloured quiver $Q$, arising from an
  $m$-cluster tilting object associated with $H$, is finite. 
\end{enumerate}
\end{thm}
\begin{proof}

(1) implies (2) and (4) implies (5) is clear. 

(2) implies (3): Suppose there are only a finite number of quivers
occurring for $m$-cluster tilted algebras associated with $H$, and let
$u$ be the maximal number of arrows between to vertices in the
quiver. Then by Lemma \ref{brlemma2}, for any two indecomposable summands
$T_1$ and $T_2$ of an $m$-cluster tilting object $T$, $\dim
\Hom_{\mathcal{C}_H^m}(T_1,T_2) < u^{2n}$, where $n$ is the number of
simple $H$-modules. Then it follows from Lemma \ref{brlemma1} that $H$
is not wild with more than $3$ simples. 

(3) implies (4): If $H$ is of finite representation type this is
clear, since we only have a finite number of indecomposables.

Next, suppose $H$ has at most two non-isomorphic simple modules. If
there is only one simple module we have $H \simeq k$, so we can assume there
are two simples. Suppose $R$ is a regular indecomposable
$H$-module. Then it follows from Lemma \ref{kQ1} that $R$ is not
rigid, i.e. $\Ext^{1}_{\mathcal{C}_{H}^{m}}(R,R) \neq 0$. Then
we also have that $\Ext^{1}_{\mathcal{C}_{H}^{m}}(R[i],R[i]) \neq 0$
for any $i \in \{1,2,...,m-1\}$. Up to $\tau$ in $\mathcal{C}_{H}^{m}$
we can assume that an $m$-cluster tilting object has a direct summand
which is a shift of a projective $H$-module, say $P[j]$. Then $P[j]$
has $m+1$ indecomposable complements. It follows that there are only a
finite number of $m$-cluster tilting objects up to $\tau$, since there
are only a finite number of choices for $P[j]$. 

Suppose $H$ is tame. By Proposition \ref{thm_lem}, an $m$-cluster
tilting object has at least one direct summand which is a shift of a
projective or injective, and hence up to $\tau$ we can assume it has
an indecomposable direct summand which is a shift of a
projective. From Lemma \ref{projectiveshift} we have that there is
only a finite number of $m$-cluster tilting objects with a shift of an
indecomposable projective $H$-module as a direct summand. 

(5) implies (6): This is clear, since mutation of $m$-cluster tilting
objects corresponds to mutation of coloured quivers.   

We have that (4) implies (1) by using Lemma
\ref{colouredquiveriso}. (6) implies (2) is trivial, and so we are
done. 
\end{proof}

We get the following corollary. 

\begin{cor}
A coloured quiver $Q$ corresponding to an $m$-cluster tilting object,
has finite mutation class if and only if $Q$ is mutation equivalent to
a quiver $Q'$, where $Q'_G$ has underlying graph Dynkin or extended
Dynkin, or it has at most two vertices, and there are only arrows of
colour $0$ and $m$ in $Q'$.
\end{cor}

\textbf{Acknowledgements:} The author would like to thank Aslak Bakke
Buan for valuable discussions and comments.

\end{document}